\documentclass[12pt,leqno,indentfirst]{article}

\usepackage{amstext}

\usepackage{amsthm}
\usepackage{amsopn}
\usepackage{amsfonts}
\usepackage{amsmath}
\usepackage{amssymb}
\usepackage{amsfonts}
\usepackage{amscd}

\newtheorem{theorem}{Theorem}[section]
\newtheorem{lemma}[theorem]{Lemma}
\newtheorem{proposition}[theorem]{Proposition}

\newtheorem{corollary}[theorem]{Corollary}

\theoremstyle{definition}

\theoremstyle{remark}
\newtheorem{remark}[theorem]{\bf Remark}

\numberwithin{equation}{section}

\newcommand{\be}{\begin{equation}}
\newcommand{\ee}{\end{equation}}
\def\q{\quad}
\def\qed{\hfill $\square$}
\def\qbar{{{\scriptstyle Q}\kern-.45em{\vrule height.41em width.035em
depth-.03em}}~}
\def\Cbar{\text{\sl C\kern-.35em{\vrule height.63em width.05em
depth-.033em}}~}
\def\cbar{{{\scriptstyle C}\kern-.41em{\vrule height.42em width.035em
depth-.03em}}~}
\def\ibid{\hbox to .5truein{\hrulefill}}
\def\IH{\text{{\rm I}\kern-.13em{\rm H}}}

\def\Z{\mathbb Z}
\def\Qbar{\mathbb Q}

\def\twoheaddown{\downarrow\kern-0.78em\raise0.25em\hbox{$\downarrow$}}
\def\headtaildown{\downarrow\kern-0.79em\raise 0.5em\hbox{$\ssize\curlyvee$}}
\def\longdownarrow{\downarrow\kern-0.73em\raise0.5em\hbox{$\vert$}}
\def\hookdownarrow{\longdownarrow\kern-.12em\raise0.5em\hbox{$^{^\cap}
$}}

\def\vs{\vskip.3cm}
\def\vsk{\vskip.5cm}

\def\noi{\noindent}

\def\pr{\prime}
\def\bs{\backslash}

\def\wh{\widehat}

\def\ol{\overline}
\def\os{\overset}
\def\ul{\underline}
\def\us{\underset}

\def\CF{{\cal F}}

\def\CR{{\cal R}}

\def\dim{\text{\rm dim}}

\def\Aut{\text{\rm Aut}}

\def\Hom{\text{\rm Hom}}
\def\id{\text{\rm id}}

\def\elra{\hbox to 2in{\rightarrowfill}}
\def\ella{\hbox to 2in{\leftarrowfill}}
\def\hrf{\hbox to 2in{\hrulefill}}
\def\hdotfill{\leaders\hbox to 1em{\hss .\hss}\hfill}

\def\be{{\pmb e}}

\def\fp{{\frak p}}

\begin{document}

\title{GALOIS STRUCTURE OF $S$-UNITS}

\author{D. Riveros and  A. Weiss}

\date{} \maketitle

\centerline{To the memory of K.W. Gruenberg}

\centerline{}
\vsk
\begin{abstract}
Let $K/k$ be a finite Galois extension of number fields with Galois
group $G,$ \, $S$ a {\it large} $G$-stable set of primes of $K,$ and
$E$ (respectively $\pmb\mu  )$ the $G$-module of $S$-units of $K,$
(resp. roots of unity).  Previous work using the Tate sequence of
$E$ and the Chinburg class $\Omega  _m$ has shown that the stable
isomorphism class of $E$ is determined by the {\it data} $\Delta 
S,\, \pmb\mu,\, \Omega  _m\,,$ and a special character $\varepsilon  $
of $H^2\big(G,\Hom(\Delta  S,\pmb\mu  )\big).$  This paper explains how to build
a $G$-module $M$ from this data which is stably isomorphic to
$E\oplus \Z  G^n,$ for some integer $n.$
\end{abstract}

\vsk
Let $K/k$ be a finite Galois extension of number fields with Galois
group $G$ and let $S$ be a finite $G$-stable set of primes of $K$
containing all archimedean primes.  Assume that $S$ is {\it large}
in the sense that it contains all ramified primes of $K/k$ and that the
$S$-class group of $K$ is trivial.  Let $E$ denote the $G$-module of
$S$-units of $K$ and $\pmb\mu  $ the roots of unity in $K.$  The purpose
of this paper is to specify the stable isomorphism class of the
$G$-module $E$ in a much more explicit way than in Theorem~B of
\cite{GW2}.

More precisely, and continuing in the notation of \cite{GW2}, we
recall that \cite{T1}, \cite{T2} defines a canonical 2-extension
class of $G$-modules, represented by Tate sequences
$$
0\to E\to A\to B\to 
\Delta  S\to 0,
$$
with $A$ a finitely generated cohomologically trivial $\Z G$-module,
$B$ a finitely generated projective $\Z G$-module and $\Delta  S$
the kernel of the $G$-map $\Z S\to \Z$ which sends every element of
$S$ to $1.$  From this 
\cite{C1} obtains the Chinburg $\Omega  (3)$-class
$$
\Omega  _m:= [A] - [B]
$$
in the locally free class group ${\rm Cl} (\Z G)\subseteq K_0 (\Z
G),$ which is an invariant of $K/k$ that is independent of $S,$ and
conjectures that $\Omega  _m$ equals the root number class in ${\rm
Cl}(\Z G).$ 

The method of \cite{GW2} analyzes the $G$-module $E$ in terms of a
{\it fixed} envelope of $\pmb\mu  .$  This is an exact sequence
\begin{equation}\label{Eq0.1}
0\to \pmb\mu  \to \pmb\omega  \to \ol{\pmb\omega}  \to 0,
\end{equation}
with $\pmb\omega  $ cohomologically trivial and $\ol{\pmb\omega}  $ the $\Z
G$-lattice obtained from $\pmb\omega  $ by factoring by its
$\Z$-torsion.  By Theorem B, the $G$-module $E$ is determined, up to
stable isomorphism, by knowledge of the $G$-set $S,$ the $G$-module
$\pmb\mu  ,$ the Chinburg class $\Omega  _m(K/k)\in \,{\rm Cl}(\Z [G]),$
and an arithmetically defined character
$$
\varepsilon  \in H^2\big(G,\Hom(\Delta  S,\pmb\mu  )\big)^*,
$$
where $\ul{\;\,}^*$ means $\Hom (\,\ul{\;\,},\Qbar/\Z).$

Let $L_1:= \Delta  G\otimes \Delta  S$ and $L_2:= \Delta  G\otimes
L_1$ with $\otimes = \otimes_{\Z}$ and diagonal action by $G.$ 
Choose the envelope $\pmb\omega  $ to be related to the Chinburg class
by the condition
\begin{equation}\label{Eq0.2}
[\pmb\omega  ] - w[\Z G] = \Omega  _m(K/k)\; \text{in}\;{\rm Cl}(\Z G),
\end{equation}
with $\vert  G\vert  w$ equal to the $\Qbar$-dimension of $\Qbar
\otimes \pmb\omega  .$  We will construct a canonical isomorphism
$H^2\big(G,\Hom(\Delta  S,\pmb\mu  )\big)^* \to H^1\big(G,\Hom(\pmb\omega 
,L_2)\big)$ so that our main result is the 

\vs\noi
{\bf Theorem.} {\it 
Let $M=M(\varepsilon  )$ denote the $G$-module in a $\Z$-split
{\rm 1}-extension
$$
0\to L_2\to M\to \pmb\omega  \to 0
$$
with extension class equal to the image $\varepsilon  ^{(1)}$ of
$-\varepsilon  $ in $H^1\big(G,\Hom(\pmb\omega  ,L_2)\big).$  Then
$E\oplus (\Z G)^n$ is stably isomorphic to $M(\varepsilon  ),$ with
$n:= (\vert  G\vert  -2)(\vert  S\vert  -1) + w$ when $G\ne 1.$ }

\vs
This improves Theorem B by explaining {\it how} its data determines
$M,$ a {\it model} for the stable isomorphism class of $E.$  The
remaining problem becomes not only to understand the ingredients
$\Delta  S,\pmb\omega  ,\Omega  _m,\varepsilon  ,n$ of the Theorem,
 but to do so in a way that improves $M$ into a better
approximation of $E.$ As a first example of this, we show how to get
a smaller $n,$ and an $M^\pr,$ in  Corollary \ref{Cor4.1}.  There is also a
continuing discussion on the relation of the Theorem 
with \cite{GW2}, including a Proposition 2.2, and especially on the
role of the distinguished character $\varepsilon,  $ in Remark
\ref{Rem4.6} and Lemma \ref{Lem4.5}.

Our proof of the Theorem,  based on \cite{GW2}, is
presented in three sections: the first recalling relevant results, the
second reformulating the Theorem  in their terms, and the
third containing a proof. The last section discusses some basic
aspects of the many new problems that arise.

\section{Review of \cite{GW2}}\label{Sec1}

Applying $\ul{\,\;}\,\otimes \Delta  S$ to the $(\Z$-split)
augmentation sequence $0\to \Delta  G\to \Z G\to \Z \to 0$ gives the
$(\Z$-split) $G$-module sequence
\begin{equation}\label{Eq1.1}
0\to L_1\to \Z G\otimes \Delta  S\to \Delta  S\to 0,
\end{equation}
with $\Z G\otimes \Delta  S$ a free $\Z G$-module, and $L_1:= \Delta
 G\otimes \Delta  S.$  Applying $\Hom(\,\ul{\;\,}\,,\pmb\mu  )$ to this
gives the exact $G$-module sequence
$$
0\to \,\Hom(\Delta  S,\pmb\mu  ) \to\,\Hom(\Z G\otimes \Delta 
S,\pmb\mu  )
\to\,\Hom(L_1,\pmb\mu  ) \to 0,
$$
inducing the connecting isomorphism in Tate cohomology
\begin{equation}\label{Eq1.2}
\partial  _1: H^1\big(G,\Hom(L_1,\pmb\mu  )\big) \to
H^2\big(G,\Hom(\Delta  S,\pmb\mu  )\big)
\end{equation}
and defining $\varepsilon  _1:=\varepsilon  \circ \partial  _1\in
H^1\big(G,\Hom(L_1,\pmb\mu  ))^*.$

Similarly, applying $\Hom(L_1,\ul{\;\,}\,)$ to our fixed envelope
\eqref{Eq0.1} of $\pmb\mu  $ and then $G$-cohomology gives the
\begin{equation}\label{Eq1.3}
\partial  ^\pr_0: \wh H\,^0\Big(G,\Hom(L_1,\ol{\pmb\omega}  )\big) \to
H^1\big(G,\Hom(L_1,\pmb\mu  )\big),
\end{equation}
and defines $\varepsilon  _0:=\varepsilon  _1\circ \partial 
^\pr_0\in \wh H\,^0\big(G,\Hom(L_1,\ol{\pmb\omega}  )\big)^*.$

We now use the isomorphism
\begin{equation}\label{Eq1.4}
\wh H\,^0\big(G,\Hom(\ol{\pmb\omega},L_1  )\big)\to
H^1\big(G,\Hom(L_1,\ol{\pmb\omega})\big)^*,
\end{equation}
from \eqref{Eq1.2} of loc.cit, that sends $[f]$ to $[f]^*$ with
$[f]^*$ represented by the element 
\newline $g\mapsto (1/\vert  G\vert  )$
trace $(f\circ g) + \Z$ of $\Hom_G(L_1,\ol{\pmb\omega}  )^*.$  It follows
that
\begin{equation}\label{Eq1.5}
\varepsilon  _0 = [f]^*\;\text{for some} \; G\text{-homomorphism}\;
f:\ol {\pmb\omega}  \to L_1\,.
\end{equation}

Extension classes in Tate cohomology are as in \S 11 of \cite{GW1}
(cf. Remark after 11.1): a $\Z$-split 1-extension $(M):0\to X\to
M\to Y\to 0$ of $G$-modules remains exact on applying $\Hom
(Y,\ul{\;\,}\,),$ and the connecting homomorphism
\begin{equation}\label{Eq1.6}
\partial  _{(M)}: \wh H\,^0\big(G,\Hom(Y,Y)\big)\to
H^1\big(G,\Hom(Y,X)\big)
\end{equation}
on its $G$-cohomology allows the definition $\xi  _{(M)} := \partial
 _{(M)}(\id_Y) \in H^1\big(G,\Hom(Y,X)\big)$ of the extension class
of $(M).$  Note that $(M)\mapsto \xi  _{(M)}$ induces a bijection
between the set of equivalence classes of $\Z$-split 1-extensions $(M)$ and
$H^1\big(G,\Hom(Y,X)\big).$

The notational deviation $L_1,\varepsilon  _1$ from the
$L,\varepsilon  $ of \cite{GW2} in \eqref{Eq1.1} is intended to
separate the role of $\varepsilon  _1\,$ which is at the centre of
the envelope focus of loc.cit. (so {\it every} $\varepsilon  $ after
the first two pages there is now $\varepsilon  _1),$ from that of
the more fundamental $\varepsilon  .$ The basic idea, only partially
realized by Theorem~B, is to use the homotopy class $[f]$ to
`reconstruct' $E:$ the formation in Proposition~5.1 of the
`homotopy' kernel $M^\pr$ of $f_0$ doesn't 
provide a description of
$M^\pr.$  This defect
is here addressed by using extension classes.

We will use, near \eqref{Eq3.5}, the notation $[L_1,N] = \wh
H\,^0\big(G,\Hom(L_1,N)\big)$ from (5.1) of \cite{GW1} to evoke the
homotopy language. Given an envelope $(C): 0\to M\to C\to L_1\to 0,$
with $\Z$-torsion $j: \pmb\mu\hookrightarrow M  ,$ applying
$\Hom(L_1,\ul{\;\,}\,)$ and $G$-cohomology gives an isomorphism
\begin{equation}\label{Eq1.7}
\partial  _{(C)}:[L_1,L_1]\to H^1\big(G,\Hom(L_1,M)\big),
\end{equation}
of right $[L_1,L_1]$-modules.  Then $\tau  _1\partial 
_{(C)}^{-1}j_*$ is in $H^1\big(G,\Hom(L_1,\pmb\mu  )\big)^*$ and we say,
following \eqref{Eq1.6} of \cite{GW2}, that $(C)$ is {\it linked} to
its $\Aut_G(\pmb\mu  )$-orbit.  This orbit is here insensitive to
the choice of $j,$ because $\Aut_G(\pmb\mu  ) =
\,\Aut(\pmb \mu)$ since $\pmb\mu  $  cyclic implies that
$\Aut(\pmb\mu)$ is abelian.

\section{Reformulation}\label{Sec2}

First, applying $\ul{\;\,}\,\otimes L_1$ to the augmentation
sequence, as in \eqref{Eq1.1}, gives a $\Z$-split $G$-module sequence
\begin{equation}\label{Eq2.1}
0\to L_2\to \Z G\otimes L_1\os{p_{_1}}\to L_1\to 0,
\end{equation}
with $\Z G\otimes L_1$\, $\Z G$-free and $L_2:=\Delta  G\otimes
L_1\,.$  Thus applying $\Hom(\pmb \omega  ,\,\ul{\;\,}\,),$ as in
\S1, and then $G$-cohomology gives the connecting isomorphism
\begin{equation}\label{Eq2.2}
\delta  _0: \wh H\,^0\big(G,\Hom(\pmb \omega  ,L_1)\to
H^1(G,\Hom(\pmb\omega  ,L_2)\big).
\end{equation}

Our reformulation  starts from the trivial observation that the
$G$-map $\pmb\omega  \to \ol{\pmb\omega}$ of \eqref{Eq1.1} induces
an equality of the functors $\Hom(\ol{\pmb\omega  },\,\ul{\;\,}\,)
\to\Hom(\pmb\omega  ,\,\ul{\;\,}\,)$ on $\Z G$-lattices $X.$  Then
\begin{equation}\label{Eq2.3}
\wh H\,^0 \big(G,\Hom(\ol{\pmb\omega  }, L_1)\big) =\wh
H\,^0\big(G,\Hom(\pmb \omega  ,L_1)\big)
\end{equation}
allows us to rewrite \eqref{Eq1.4} as an isomorphism
\begin{equation}\label{Eq2.4}
\wh H\,^0\big(G,\Hom(\pmb\omega  ,L_1)\big) \to \wh
H\,^0\big(G,\Hom(L_1, \ol{\pmb\omega}  )\big)^*
\end{equation}
that sends $[h]$ to $[h]^*$ with $[h]^*$ represented by the element
$g\mapsto (1/\vert  G\vert  )\,{\rm trace}\,(\ol h \circ g) + \Z,$
of $\Hom_G(L_1, \ol{\pmb\omega  })^*.$   It follows that
\begin{equation}\label{Eq2.5}
\varepsilon  _0 = [h]^* \;\text{\rm for some}\; h\in \Hom(\pmb\omega
 , L_1)^G. \end{equation}

We now define the isomorphism before 
the Theorem  of the introduction to be the composition of
the isomorphisms
\begin{equation}\label{Eq2.6}
\begin{aligned}
H^2\big(G,\Hom(\Delta  S,\pmb\mu  )\big)^* &\to
H^1\big(G,\Hom(L_1,\pmb\mu  )\big)^* \to \wh
H\,^0\big(G,\Hom(L_1,\ol{\pmb\omega}  )\big)^*\\
&\gets \wh H\,^0 \big(G,\Hom(\pmb\omega  ,L_1)\big) \to
H^1\big(G,\Hom(\pmb\omega  , L_2)\big)
\end{aligned}
\end{equation}
of \eqref{Eq1.2}$^*$, \eqref{Eq1.3}$^*$, \eqref{Eq2.4},
\eqref{Eq2.2}, and observe that it takes $-\varepsilon  $ to
$-\delta  _0([h]).$

It follows that $\varepsilon  ^{(1)} = -\delta  _0([h])$ in the
statement of the Theorem of the introduction, which is therefore
equivalent to the following reformulation.

\begin{theorem}\label{Thm2}
Let $[h]\in \wh H\,^0 \big(G,\Hom(\pmb \omega  ,L_1)\big)$ be the
image of $\varepsilon  $ under the composite of the first three maps
in \eqref{Eq2.6}, and let $\delta  _0$
 be the last map of that composite, as in \eqref{Eq2.2}.  Let $M$ be the
$G$-module in a $\Z$-split {\rm 1}-extension
$$ 
0 \to L_2 \to M\to \pmb\omega  \to 0
$$
with extension class equal to  $-\delta  _0([h])$ in
$H^1\big(G,\Hom(\pmb\omega  ,L_2)\big).$  Then $E\oplus (\Z G)^n$ is
stably isomorphic to $M,$ with $n:= (\vert  G\vert  -2)(\vert  S\vert
 -1) + w$ when
$G\ne 1.$

In particular, the class $\varepsilon  $ and the
extension class of $M(\varepsilon  )$ determine each other uniquely.
\end{theorem}

The envelope focus of \cite{GW2} overemphasizes $\varepsilon  _1$
for our purposes.  We eventually need to restate Theorem A in terms
of $\varepsilon:$ see Remark \ref{Rem4.6}.  The connection between
$\varepsilon  $ and $\varepsilon  _1$ is a consequence of the
relationship between Tate sequences and Tate envelopes, or, more
precisely, between the Tate canonical class $\alpha  _3 \in
H^2\big(G,\Hom(\Delta  S,E)\big)$ and Tate envelopes.  Thus,
following the last four paragraphs of Tate's proof of Theorem~5.1 of
Chapter~2 in \cite{T2}, we select a {\it special} Tate sequence
representing $\alpha  _3$ and {\it define} the Tate envelope to be
the left half of this special Tate sequence.

\begin{proposition}\label{Prop2.7}
A Tate envelope $0\to E\to A\to L_1 \to 0$ has
$$
\Omega  _m = A-(\vert  S\vert   -1 )[\Z G] \q\text{\rm in}\q {\rm
Cl}(\Z G).
$$
\end{proposition}

\begin{proof}
We specialize Tate's initial exact sequence by selecting the one
\begin{equation}\label{Eq2.7}
0 \to L_2 \to B^\pr \to B \to \Delta  S \to 0,
\end{equation}
obtained by splicing \eqref{Eq1.1} and \eqref{Eq2.1}; Tate's 
first paragraph ends with isomorphisms
$$
\wh H\,^r\big(G,\Hom(L_2,E)\big) \simeq \wh
H\,^{r+2}\big(G,\Hom(\Delta  S,E)\big),
$$
for all $r\in \Z,$ in our notation.  The second paragraph chooses
$\alpha  \in\Hom_G(L_2\,,E)$ corresponding to $\alpha  _3\in
H^2\big(G,\Hom(\Delta  S,E)\big)$ and deduces, from his (5.2), that
$\alpha  $ induces isomorphisms $\wh H\,^r(G,L_2) \to \wh
H\,^r(G,E),$ for all $r;$ the third paragraph extends $\alpha  $
to a surjective $\alpha  :L_2\oplus F\to E,$ with $F$ free, and
replaces $L_2\to B^\pr$ in \eqref{Eq2.7} by $L_2\oplus F\to B^\pr
\oplus F$ to get a new \eqref{Eq2.7} and the exact sequence $0\to
\,\ker (\alpha  )\to L_2\oplus F\to E\to 0.$  The fourth paragraph
deduces that $\ker (\alpha  ),$ and thus $A:= (B^\pr \oplus F)/\ker
(\alpha  ),$ is cohomologically trivial.  Combining with the new
\eqref{Eq2.7} gives the Tate sequence $0\to E\to A\to B\to\Delta 
S\to 0,$ the left half $0\to E\to A\to L_1\to 0$ of which is our
Tate envelope.

Now $B=\Z G\otimes \Delta  S\simeq (\Z G)^ {\vert  S\vert 
-1}$
 implies
that $\Omega  _m =[A]-[B]= [A] - ( \vert  S\vert  -1)[\Z G].$
\end{proof}

\section{Proof of the reformulated Theorem}\label{Sec3}
The proof is now straightforward.  We assume that $G\ne 1$ (since
the $G=1$ case, while true with the obvious interpretation, is
trivial), and start by fixing an envelope
$$
0\to \pmb\mu  \to \pmb \omega  \to \ol{\pmb\omega  } \to 0,
$$
satisfying \eqref{Eq0.1} and \eqref{Eq0.2}.  The existence of such
an $\pmb \omega  $ follows from \eqref{Eq2.1} in \cite{GW1} and
(39.12), (32.13) in \cite{CR}:  start with any envelope $0\to
\pmb\mu  \to C\to \ol C \to 0,$ define $c$ by $\vert  G\vert  c
=\,\dim \, \Qbar \otimes C,$ and observe that $\Omega  _m - ([C] -
c[\Z G]) = [P] - [\Z G]$ in ${\rm Cl}(\Z G),$ for some projective
$\Z G$-module $P$ with $\dim\, \Qbar \otimes P = \vert  G\vert  ,$
hence $C^\pr := C\oplus P$ gives an envelope $0\to \pmb\mu  \to C^\pr
\to \ol{(C^\pr)} \to 0$ with $\Omega  _m = [C^\pr] -c^\pr [\Z G],$
as required.

Letting $[h],$ with $h\in \,\Hom(\pmb \omega  ,L_1)^G,$ be as in the
assertion of  Theorem \ref{Thm2}, define 
\newline $\eta  :(\Z G\otimes
L_1)\oplus \pmb\omega  \to L_1$ by $\eta  \big((x,y)\big) = p_1(x) +
h(y),$ and form the big diagram
\begin{equation}\label{Eq3.1}
\begin{array}{cccccccccc}
&&0&&0\\
&&\downarrow &&\downarrow\\
0 &\longrightarrow &L_2 &\lhook\joinrel\longrightarrow &\ker\,(\eta  )
&\os{p_0}\longrightarrow &{\pmb\omega  } &\longrightarrow &0\\
&&\hookdownarrow &&{}\,\hookdownarrow&&\Vert  \\
0 &\to&\Z G\otimes L_1 &\hookrightarrow &(\Z G\otimes L_1) \oplus
\pmb\omega &\to &\pmb\omega   &\to &0  \\
&&{}\;\longdownarrow{} ^{{p_1}} &&{}\;\longdownarrow{}\,^{\eta}
\\
&&L_1 &\os
{\hbox to .25in{\hrulefill}}{\hbox to.1in{\hrulefill}} &L_1\\
&&\downarrow &&\downarrow\\
&&0 && 0
\end{array}
\end{equation}
as follows: start from the commutative square containing $p_1$ and
$\eta  ,$ use it to form the bottom two rows with the additional map
sending $(x,y)$ to $y,$ and then get the top row by taking kernels,
and using \eqref{Eq2.1} as the first column.  We put
$M:=\,\ker\,(\eta  )$ and focus first on the column and then on the
row containing $M.$

Now let $0\to E\to A\to L_1\to 0$ be a fixed Tate envelope, and form
the envelope
\begin{equation}\label{Eq3.2}
0 \to (\Z G)^n \oplus E \to (\Z G)^n\oplus A\to L_1\to 0,
\end{equation}
from it by adding $(\Z G)^n = (\Z G)^n.$  This is an envelope with
$\Z$-torsion $\pmb\mu  $ and lattice $L_1\,,$ as is the middle column
\begin{equation}\label{Eq3.3}
0\to M\to (\Z G\otimes L_1)\oplus \pmb\omega  \to L_1\to 0,
\end{equation}
of \eqref{Eq3.1}.  We now apply Theorem 4.7 of \cite{GW2} to show
that the left ends of these envelopes are stably isomorphic.  This
requires two conditions to be verified.

The quicker condition to check is that $[(\Z G\otimes L_1) \oplus
\pmb\omega  ]$ is equal to $[(\Z G)^n\oplus A]$ in ${\rm Cl}(\Z G).$
 Now $\Z G\otimes L_1\simeq (\Z G)^{(\vert  G\vert  -1)(\vert  S\vert 
-1)},$ because it's $\Z G$-free; and \eqref{Eq0.2} applies to
$[\pmb\omega  ],$ hence
\newline  $[(\Z G\otimes L_1)\oplus \pmb\omega  ] =
(\vert  G\vert  -1)(\vert  S\vert  -1)[\Z G] + w[\Z G] +\Omega 
_m\,.$  Similarly, the second expression equals $n[\Z G] + (\vert 
S\vert  -1)[\Z G] + \Omega  _m\,,$ by Proposition \ref{Prop2.7}.  These agree by
the choice of $n.$

The other condition is that both of these envelopes are linked to
the {\it same} $\Aut_G(\pmb \mu  )$-orbit on
$H^1\big(G,\Hom\,(L_1,\pmb\mu  )\big)^*,$ which we will show is
$\varepsilon  _1\,\Aut_G(\pmb \mu  ).$

First, by definition, the Tate envelope is linked to $\tau 
_1\partial  _{(A)}^{-1}j_*;$ and with $j:\pmb\mu  \hookrightarrow E$ the inclusion,
which is $t_Ej_*$ by definition of the trace character $t_E$ in \S
7, i.e the `restriction' $\varepsilon  _1$ of $t_E$ to
$H^1\big(G,\Hom\,(L_1,\pmb\mu  )\big).$  To get the same
conclusion for the envelope \eqref{Eq3.2}, consider the commutative
diagram defined by inclusion of the Tate envelope into
\eqref{Eq3.2}, and apply $\Hom\,(L_1,\,\ul{\;\,}\,)$ and
$G$-cohomology to get the commutative square, with all maps
isomorphisms, inside the commutative diagram
$$
\begin{array}{rccclc}
&\kern-1em H^1\big(G,\Hom\,(L_1\,,E)\big) &\os{\partial  _{(A)}}\longleftarrow
 &\kern-.3em\wh H\,^0\big(G,\Hom\,(L_1,L_1)\big)\\
{}\q\q\os{j_*\,}\nearrow &&&&\os{^{\tau  _1}}\searrow\\
H^1\big(G,\Hom\,(L_1,\pmb\mu  )\big) &\kern-1em\longdownarrow \,^{\simeq}
&&\kern-.3em\Vert  &{}\q \Qbar/\Z\\
{}\q\q\us{j_*}\searrow &&&&\us{\,\tau  _1}\nearrow\\
&\kern-1em H^1\big(G,\Hom\,(L_1,(\Z G)^n \oplus
E)\big)&\us{\partial  _{((\Z G)^n\oplus A)}}\leftarrow
&\kern-.3em\wh
H\,^0\big(G,\Hom\,(L_1,L_1)\big)
\end{array}
$$

\vs\noi
with left triangle from composing the inclusions $\pmb\mu 
\hookrightarrow E$
and $E\hookrightarrow (\Z G)^n \oplus E.$  The top composite from
$H^1\big(G,\Hom\,(L_1,\pmb\mu  )\big)$ to $\Qbar/\Z$ is equal to
$\varepsilon  _1\,,$ by the first sentence of this paragraph, hence
so is the bottom one.

Next, to see that the envelope \eqref{Eq3.3} is linked to $\varepsilon 
_1\,,$ consider the commutative diagram
$$
\begin{array}{rcccccl}
0 \rightarrow &\pmb\mu  &\rightarrow &\pmb\omega    &\rightarrow
&\ol{\pmb\omega    } &\rightarrow 0\\
&{}\,\longdownarrow \,^{j^\pr} &&{}\,\longdownarrow\, ^k
&&{}\,\longdownarrow\,^{\ol h}\\
0 \rightarrow &M &\hookrightarrow &C &\os \eta\rightarrow &L_1 &\rightarrow 0
\end{array}
$$
with top row the envelope $(\pmb\omega  )$ of \eqref{Eq0.1},
\eqref{Eq0.2}, bottom row the vertical envelope $(C)$ of
\eqref{Eq3.1} with $C= (\Z G\otimes L_1)\oplus \pmb\omega  ,$ and
$k(y) = (0,y)$ for all $y\in \pmb\omega  .$ 
Here, forming the right square first defines $j^\pr.$
Applying
$\Hom\,(L_1\,,\,\ul{\,\;}\,)$ and
 $G$-cohomology gives the commutative
square
\begin{equation}\label{Eq3.5}
\begin{array}{lcc}
[L_1,\ol{\pmb\omega  }] &\os{\partial  _{(\pmb\omega)}}
\longrightarrow &H^1\big(G,\Hom\,(L_1,\pmb\mu  )\big)\\
{}\q \longdownarrow\, ^{[\id_{L_1},\ol h]} &&\longdownarrow\,
^{(j^\pr)_*}\\
{}[L_1,L_1] 
&\us{\partial  _{(C)}}\longrightarrow
&H^1\big(G,\Hom\,(L_1,M)\big)
\end{array}
\end{equation}
with horizontal isomorphisms and $(C)$ linked to $\tau  _1\partial 
_{(C)}^{-1}(j^\pr)_* \in H^1\big(\Hom\,(L_1,\pmb\mu  )\big)^*,$ by
the definition \eqref{Eq1.7}, with $\tau  _1:[L_1,L_1] \to
\Qbar/\Z.$  Our hypothesis on $[h]$ implies the $[\ol h]^*
=\varepsilon  _1\partial  ^\pr_0$ by \eqref{Eq2.5}, \eqref{Eq1.5}
and \eqref{Eq1.3}, with $\partial ^ \pr_0 = \partial  _{(\pmb\omega 
)}\,,$ i.e. $[\ol h]^* =\varepsilon  _1\partial  _{(\pmb\omega 
)}\,.$

Now, quoting \cite{GW2}, $\varepsilon  _1\in H^1\big(G,\Hom\,(L_1,\pmb\mu
 )\big)^*$ implies that $\varepsilon  _1 = \tau  _1\theta  $ for
some right $[L_1,L_1]$-homomorphism $\theta 
:H^1\big(\Hom\,(L_1,\pmb\mu  )\big)\to [L_1,L_1],$ by \eqref{Eq1.3}.
 Then $\theta  \partial  _{(\pmb\omega  )}$ is a right
$(L_1,L_1]$-homomorphism: $[L_1,\ol{\pmb\omega  }]\to [L_1,L_1]$ so
that $[\ol h]\in [\ol{\pmb\omega  }, L_1]$ having $[\ol h]^*=
\varepsilon  _1\partial  _{(\pmb\omega  )} = \tau  _1\theta 
\partial  _{(\pmb\omega  )}\,,$ by the previous paragraph, implies
that $\theta  \partial  _{(\pmb\omega  )} = [\id_{L_1}\,,\ol h],$ by
\eqref{Eq1.4}.

Combining with \eqref{Eq3.5} above gives $\tau  _1\partial 
_{(C)}^{-1}(j^\pr)_* =\tau  _1[\id_{L_1}, \ol h]\partial 
_{(\pmb\omega  )}^{-1} = \tau _1\theta  =\varepsilon  _1\,,$ as
required.

Finally, we must show that the top row
$$
(M) : 0\to L_2\hookrightarrow M\to \pmb\omega  \to 0
$$
of the big diagram \eqref{Eq3.1} has extension class $-\delta 
_0([h]),$ in the notation of \eqref{Eq1.6}.

To get a 1-cocycle representing $-\delta  _0([h]),$ one applies
$\Hom\,(\pmb\omega  ,\,\ul{\,\;}\,)$ to \eqref{Eq2.1}, getting the
exact sequence $0\to \Hom\,(\pmb\omega  ,L_2)\to \Hom\,(\pmb\omega 
,\Z G\otimes L_1)\to \Hom\,(\pmb\omega  ,L_1)\to 0,$ chooses a
pre-image of $h$ in $\Hom\, (\pmb\omega  ,\Z G\otimes L_1),$ say the
map $1\otimes h$ taking every $y\in \pmb\omega  $ to $1\otimes
h(y),$ and then forms the 1-cocycle $g\mapsto (1\otimes h) -
g(1\otimes h)$ (with $g\in G)$ taking values in $\Hom\,(\pmb\omega 
,L_2),$ namely 
${}[(1\otimes h)-g\!\cdot(1\otimes h)](y) 
= (1\otimes h)(y) -g\!\cdot (1\otimes h)(g^{-1}y)
= 1\otimes h(y)-g\!\cdot\big(1\otimes h(g^{-1}y)\big)
= 1\otimes h(y)-g\otimes g\!\cdot\! h(g^{-1}y)
= 1\otimes h(y) -g\otimes h(y) = (1-g)\otimes h(y)\in 
\Delta  G\otimes L_1=L_2\,.$

On the other hand, the extension class $\xi  _{(M)} $ of $(M)$ is,
by definition, obtained from $(M) $ by applying $\Hom\,(\pmb\omega 
,\,\ul{\,\;}\,)$ to $(M),$ getting $0\to \Hom\,(\pmb\omega 
,L_2)\to\Hom\,(\pmb\omega  ,M)\to \Hom\,(\pmb\omega  ,\pmb\omega 
)\to 0,$ lifting $\id_{\pmb\omega  }$ to some $s\in
\Hom\,(\pmb\omega  ,M),$ and forming the class of the 1-cocycle
$g\mapsto gs-s$ with 

\newpage
\noi
values in $\Hom\,(\pmb\omega  ,L_2).$  Setting
$s(y) = \big(-1\otimes h(y),y\big)$ works, since $\eta \big(s(y)\big) 
= p_1\big(-1\otimes h(y)\big) + h(y) =0$ and
$p_0\big(s(y)\big) = y.$ Now $(gs-s)(y) = g\big(-1\otimes
h(g^{-1}y), g^{-1}y\big) - \big(-1\otimes h(y),y\big) =
\big(-g\otimes h(y),y\big) + \big(1\otimes h(y),-y\big) =
\big((1-g)\otimes h(y),0\big),$ which is the image of $(1-g)\otimes
h(y)\in L_2\,.$  This agrees with the 1-cocycle of the previous
paragraph.
\qed

\section{Discussion}\label{Sec4}

We begin with a consequence of the Theorem, for which
we prepare with a naturality property of the Gruenberg resolution. 
We start with a subset, of $d$ elements $g_i$ of $G\bs\{1\},$ which
{\it generates} $G,$ form the free group $F$ on $x_i\,,$ \,$1\le
i\le d,$ and define the relation module $R_d$ by the exact sequence 
\begin{equation}\label{Eq4.1}
0\to R_d\to \Z G\otimes_{\Z F}\Delta  F\to \Delta  G\to 0
\end{equation}
(cf. \cite{HS} p.~199 and 218). Here the middle term is $\simeq (\Z G)^d$
since $\Delta  F$ is $\Z F$-free on the $(x_i-1)$'s, and the right
map sends the $\Z G$-basis $1\otimes_F (x_i-1)$ to $g_i-1.$

In the special case $d=\vert  G\vert  -1,$ write $\CR, \CF$ for
$R_d,\,F$ respectively. For general $d,$ the inclusion $F\to \CF$
induces a map from the relation sequence for $R_d$ to $\CR,$ which
on middle terms is an inclusion of the respective $\Z G$-bases so
has cokernel $\simeq (\Z G)^{\vert  G\vert  -1-d},$ yielding the
exact sequence $0\to R_d \to \CR \to (\Z G)^{\vert  G\vert  -1-d}\to
0$ on the left terms.

Similarly, the relation module sequence for $\CR$ maps to the exact
sequence obtained by applying $\ul{\,\;}\, \otimes \Delta  G$ to the
augmentation sequence, with middle map matching  $\Z G$-bases by
$1\otimes_{\cal F} (x_i-1)\mapsto 1\otimes (g_i-1),$ inducing an
isomorphism $\CR\to \Delta  G\otimes \Delta  G.$  This implies that
\begin{equation}\label{Eq4.2}
0\to R_d \os\beta  \rightarrow \Delta  G\otimes \Delta  G\to (\Z
G)^m \to 0
\end{equation}
is exact with an  explicit map $\beta  $ and $m=\vert  G\vert 
-1-d,$ when $G\ne 1.$

Let $d(G)$ be the minimal number of generators of $G,$ and set
$R:=R_{d(G)}\,,$ to state the

\begin{corollary}\label{Cor4.1}
There is an explicit $G$-homomorphism $\beta  ^\pr :R\otimes \Delta 
S\to L_2$ so that the induced isomorphism $\beta  ^\pr
_*: H^1\big(G,\Hom(\pmb \omega  ,R\otimes \Delta  S)\big)\to
H^1\big(G,\Hom\,(\pmb \omega  ,L_2)\big)$ has the following
property: let $M^\pr$ be the $G$-module in a $\Z$-split {\rm 1}-extension
$$
0\to R\otimes \Delta  S\to M^\pr \to \pmb\omega  \to 0
$$
with extension class mapping to $\varepsilon  ^{(1)}$ under $\beta 
^\pr_*.$  Then $E\oplus (\Z G)^{n^\pr}$ is stably isomorphic to
$M^\pr$ with $n^\pr = \big(d(G) -1\big)(\vert  S\vert  -1) + w$ when
$G\ne 1.$
\end{corollary}

\begin{proof}
By $L_2 = \Delta  G\otimes L_1 =\Delta  G\otimes (\Delta  G\otimes
\Delta  S)\simeq (\Delta  G\otimes \Delta  G)\otimes \Delta  S,$ applying
$\ul{\,\;}\, \otimes \Delta  S$ to \eqref{Eq4.2} gives the exact
sequence
\begin{equation}\label{Eq4.3}
0\to R\otimes \Delta  S\os{\beta  ^\pr}\to L_2 \to (\Z G)^{n-n^\pr} \to 0,
\end{equation}
defining $\beta  ^\pr.$  This follows from $(\Z G)^m \otimes \Delta 
S\simeq (\Z G)^{m(\vert  S\vert  -1)}$ with $m(\vert  S\vert 
-1)=n-n^\pr.$

Now the extension class of the 1-extension $(M^\pr)$ has the
property that its pushout along $\beta  ^\pr$ has extension class
$\varepsilon  ^{(1)}$ so there is a commutative diagram
$$
\begin{array}{rccccl} 
0\rightarrow &R\otimes \Delta  S &\rightarrow &M^\pr &\rightarrow
&\pmb\omega  \to 0\\
&\longdownarrow\, ^{\beta  ^\pr}&&\longdownarrow &&\Vert  \\
0\rightarrow &L_2 &\rightarrow &M&\rightarrow &\pmb\omega 
\rightarrow 0.
\end{array}
$$
Since $\beta  ^\pr$ has cokernel $(\Z G)^{n-n^\pr}$ so does the middle
arrow, hence there's an exact sequence $0\to M^\pr \to M\to (\Z
G)^{n-n^\pr} \to 0.$  Thus, by the Theorem, $E\oplus (\Z G)^n\approx
M\approx M^\pr \oplus (\Z G)^{n-n^\pr},$ which implies that $E\oplus
(\Z G)^{n^\pr} \approx M^\pr.$
\end{proof}

\begin{remark}\label{Rem4.4}
$R$ has no non-zero projective summand if $G$ is solvable or, more
generally, when $G$ has generation gap $=0$ (cf.\,(24) in \cite{G}),
in which case we cannot expect bigger $\Z G$-free summands from
the above approach.  
Note that $R_d$ is determined up to stable isomorphism by $d,$ as
follows from \eqref{Eq4.1} by Schanuel's lemma.
 Corollary \ref{Cor4.1} is a first step
toward the important goal of excising  as many $\Z G$-free
summands of $M$ as explicitly as possible.  There are many aspects
of this problem but still no
systematic approach.
\end{remark}

There has been considerable work on
Chinburg's conjecture as a special case of the  
Equivariant Tamagawa Number Conjecture; a recent reference is
\cite{B} (cf. Corollary 2.8 and Remark 2.9).  Since Chinburg's
conjecture predicts that $\Omega  _m=0$ whenever $G$ has no
irreducible symplectic representation (cf. \S3 of \cite{C2}), an
envelope $\pmb\omega  $ of $\pmb \mu$ with $[ \pmb\omega  ] -
w[\Z G]=0$ and $w =d(G) $ (cf. \cite{R}) is a useful ingredient
for examples. 

On the other hand, the condition \eqref{Eq0.2} on $\pmb\omega  $ could be replaced in
the Theorem by
$$
[\pmb\omega  ] - w[\Z G] \equiv \Omega_m\mod B[\varepsilon  _1],
$$
as the appeal to Theorem 4.7 of \cite{GW2} in its proof 
shows.  This shows that the full strength of Chinburg's conjecture
may not be needed.

\begin{remark}\label{Rem4.6}
The emphasis on $\varepsilon  _1$ in \cite{GW2} comes from the
envelope focus. In particular, Theorem~A for $\varepsilon  _1$ is
proved by this method, but its statement depends on the local and
global invariant maps on $H^2,$ where $\varepsilon  $ becomes more
central. Theorem~A can be translated from $\varepsilon  _1$ to
$\varepsilon  $ by using the formalism of \cite{T1}, in the
direction of the last paragraph of the Remark on p.~971 of
\cite{GW2}.
\end{remark}

More precisely, let $\Hom\big((\Z S),(J)\big)$ be the $G$-module
consisting of all triples $(f_1,f_2,f_3)$ of $\Z$-homomorphisms so
that the diagram
$$
\begin{array}{ccccccccc}
0 &\to &\Delta S  &\to &\Z S &\to &\Z &\to &0\\
&&^{f_1}\!\!\longdownarrow\,{} &&^{f_2}\!\!\longdownarrow\,{}
&&^{f_3}\!\!\longdownarrow\,{}\\
0 &\to &E &\to &J &\to &C_K &\to &0
\end{array}
$$
commutes. This leads to an exact sequence

\vs\noi
\hbox{\small{$
0\!\!\to\!\! H^2\big(G,\Hom((\Z S),(J))\big)\!\to\! H^2\big(G,\Hom(\Delta 
S,E)\big)\!\oplus\! H^2\big(G,\Hom(\Z S,J)\big)\! \to\!
H^2\big(G,\Hom(\Delta  S,J)\big)\!\!\to\!\! 0
$}}

\vs\noi
allowing us to study the trace character $T_E:H^2\big(G,\Hom(\Delta 
S,E)\big)\to \Qbar/\Z$ defined by dimension shifting $t_E$ using the
exact sequence \eqref{Eq1.1}.  This implies that $\varepsilon 
=T_E\circ j_*,$ with $j:\pmb \mu  \hookrightarrow E$ the inclusion, but now the
point is that $T_E$ can be described in terms of the $H^2$-sequence
above without further dimension shifting.

More precisely, given $x\in H^2\big(G,\Hom(\Delta  S,E)\big),$ there
exists $y\in H^2\big(G,\Hom(\Z S,J)\big)$ so  $(x,y)$ maps to
$0$ in $H^2\big(G,\Hom(\Delta  S,E)\big),$ hence there is a unique $T\in
H^2\big(G,\Hom\big((\Z S),(J)\big)\big)$ mapping to $(x,y).$  Taking a 2-cocycle of triples
representing $T$ and projecting on the third component gives a
2-cocycle defining $z\in H^2\big(G,\Hom(\Z,C_K)\big).$  Then (cf.
\cite{R})
\begin{equation}\label{Eq4.4}
T_E(x) =\;\text{inv}(z) -\sum_{{\fp}\in
S_*}\,\text{inv}_{\fp}(y_{\fp}),
\end{equation}
where $S_*$ is a transversal to the $G$-orbits on $S,y_{\fp} =
k_{\fp}(\text{res}\;y)i_{\fp}$ with $k_{\fp}:J\to K^\times_{\fp}$ the
projection and
$i_{\fp}:\Z \to \Z S$ with $i_{\fp}(1)={\fp}.$

This description has the weakness that the existence of $y$
apparently depends on the vanishing of
$H^3(G,J).$  This situation is improved by the

\begin{lemma}\label{Lem4.5}
The map $H^2\big(G,\Hom(\Z S,J)\big)\to H^2\big(G,\Hom(\Delta 
S,J)\big)$ has a special splitting.
\end{lemma}

\begin{proof}
The $S$-idele group $J$ is a finite product, over ${\fp}\in S_*\,,$ of components
$V_{\fp}:= \prod_{\frak q}K^\times_{\frak q}\,,$ with ${\frak q}$ running through the
$G$-orbit of ${\fp},$ up to a large cohomologically trivial
component of unit ideles.  So it suffices to show that
$H^2\big(G,\Hom(\Z S,V_{\fp})\big)\to H^2\big(G,\Hom (\Delta 
S,V_{\fp})\big)$ is split for each ${\fp}.$ 

If $H$ is a subgroup of $G,$ and $B$ any $H$-module, define the
coinduced $G$-module $\text{coind}\,(B),$ from $H$ to $G,$ to be
$\Hom_{\Z H}(\Z G,B)$ with $g\in G$ acting on elements $\varphi  $
by $(g\varphi  )(z)=\varphi  (zg)$ for all $z\in \Z G$ (cf VII, \S5
of \cite{S}).  If $D$ is any $\Z G$-lattice, viewing
$\Hom\big(D,\text{coind}\,(B)\big)$ as $G$-module and
$\Hom(\text{res}\, D,B)$ as $H$-module, both by diagonal action,
then there are natural Shapiro isomorphisms
$$
H^n\big(G,\Hom(D,\text{coind}\, B)\big) \to
H^n\big(H,\Hom(\text{res}\, D,B)\big).
$$
Take $H=G_{\fp}\,,$ \, $B=K^\times_{\fp}$ and identify
$\text{coind}\,K^\times_{\fp}$ with $V_{\fp}\,,$ via $\varphi 
\mapsto \prod_t\big(t\!\cdot\varphi  (t^{-1}\big),$ with $t$ a choice
of representatives of $G/G_{\fp}\,.$  This choice doesn't matter,
since $(th)\!\cdot \varphi  \big((th)^{-1}\big) =
t\!\cdot\big(h\!\cdot \varphi  (h^{-1}t^{-1})\big) = t\!\cdot\varphi
 (t^{-1})$ for $h\in G_{\fp}\,.$  The map is bijective, since the
components $tK^\times_{\fp}$ of $V_{\fp}$ are disjoint, and is a
$G$-homomorphism because $g\big(\prod_t(t\!\cdot\varphi  (t^{-1})\big) =
\prod_t\big((gt)\!\cdot \varphi  (gt)^{-1}g)\big) =
\prod_t\big((gt)\!\cdot (g\varphi  )(gt)^{-1})\big)
=\prod_t\big(t\!\cdot(g\varphi  )(t^{-1})\big).$

This identifies our map of the first paragraph with the top row of
the commutative square
\begin{equation}\label{Eq4.5}
\begin{array}{ccc}
H^2\big(G,\Hom(\Z S,\text{coind}\, K^\times_{\fp})\big)
&\os{a^*}\to &H^2\big(G,\Hom(\Delta  S,\text{coind}\, K^\times_{\fp})\big)\\
^{sh}\!\!\longdownarrow\,{} &&^{sh}\!\!\longdownarrow\,{}\\
H^2\big(G_{\fp},\Hom(\text{res}\, \Z S,K^\times_{\fp})\big)
&\os{a^*}\to &H^2\big(G_{\fp}\,,\Hom(\text{res}\,\Delta 
S,K^\times_{\fp})\big),
\end{array}
\end{equation}
with vertical isomorphisms, and horizontal maps from $0\to \Delta 
S\os{a}\hookrightarrow \Z S\os{a^\pr}\to \Z\to 0.$  This exact sequence is $G_{\fp}$-split, by
the $G_{\fp}$-map $\lambda  _{\fp}: d\mapsto d-a^\pr(d){\fp}$ having
$\lambda  _{{\fp}}\circ a =\,\text{id}_{\Delta  S}\,.$  Thus
$\lambda  _{\fp}$ induces $H^2\big(\Hom(\text{res}\, \Delta 
S,K^\times_{\fp})\big)\to H^2\big(G_{\fp}\,,\Hom(\text{res}\, \Z
S,K^\times_{\fp})\big)$ splitting the bottom $a^*$ of the
commutative square, which then completes the argument.
\end{proof}

\vsk

\vsk
\noi
{\small\begin{tabular} {ll}\\
A. Weiss &\q\q\q\q\q\q\q D. Riveros\\
Dept. of Mathematical &\q\q\q\q\q\q\q
Dept. of Mathematical  \\
{} \q \& Statistical Sciences &\q\q\q\q\q\q\q \q
 \& Statistical Sciences\\
University of Alberta &\q\q\q\q\q\q\q
University of Alberta\\
Edmonton, Alberta 
&\q\q\q\q\q\q\q Edmonton, Alberta\\
Canada \q T6G 2G1 &\q\q\q\q\q\q\q
Canada \q T6G 2G1\\
weissa@ualberta.ca &\q\q\q\q\q\q\q riverosp@ualberta.ca\\
\end{tabular}
\end{document}